\documentclass[sigconf]{acmart}

\copyrightyear{2026}
\acmYear{2026}
\setcopyright{acmlicensed}
\setcctype{by}
\acmConference[ISSAC '26]{51st International Symposium on Symbolic and Algebraic Computation}{July 13--17, 2026}{Oldenburg, Germany}
\acmBooktitle{51st International Symposium on Symbolic and Algebraic Computation (ISSAC '26), July 13--17, 2026, Oldenburg, Germany}
\acmDOI{10.1145/3815436.3815464}
\acmISBN{979-8-4007-2595-1/2026/07}

\acmISBN{978-1-4503-XXXX-X/2018/06}

\usepackage{amsmath}
\usepackage{amsthm}
\usepackage{amsfonts}
\usepackage{mathrsfs} 
\usepackage{graphicx}
\usepackage{color,xcolor}
\usepackage{bm}

\makeatletter
\newcommand\xleftrightarrow[2][]{%
  \ext@arrow 9999{\Leftrightarrowfill@}{#1}{#2}}
\newcommand\longleftrightarrowfill@{%
  \arrowfill@\leftarrow\relbar\rightarrow}
\makeatother

\usepackage{tikz}
\usetikzlibrary{arrows.meta, positioning}

\usepackage{hyperref}
\usepackage{cleveref}

\usepackage{algorithm,algorithmic}

\usepackage{stmaryrd} 

\newcommand{\R}{\mathbb{R}} 
\newcommand{\N}{\mathbb{N}} 
\newcommand{\sym}{\mathbb{S}} 
\newcommand{\PP}{\mathcal{P}}
\newcommand{\Qc}{\mathcal{Q}}
\renewcommand{\span}[1]{{\mathrm{span}(#1)}} 
\newcommand{\calL}{\mathcal{L}} 

\newtheorem{theorem}{Theorem}[section]
\newtheorem{lemma}[theorem]{Lemma}
\newtheorem{proposition}[theorem]{Proposition}
\newtheorem{corollary}[theorem]{Corollary}

\newtheorem{definition}[theorem]{Definition}

\newtheorem{remark}[theorem]{Remark}

\newtheorem{example}[theorem]{Example}

\DeclareMathOperator{\val}{val}
\def\QQ{\ensuremath{\mathbb{Q}}}
\def\ZZ{\ensuremath{\mathbb{Z}}}
\newcommand{\OK}{\mathcal{O}_K}
\def\diag{\mathrm{diag}}

\newcommand{\GL}{\mathrm{GL}}

\begin{document}

\title{On semidefinite-representable sets over valued fields}

\author{Corentin Cornou}
\orcid{0009-0009-0085-4288}
\affiliation{
    \institution{ENS Paris-Saclay}
    \city{Gif-sur-Yvette}
    \country{France}
}

\author{Simone Naldi}
\orcid{0000-0002-4556-6935}
\affiliation{Universit\'e de Limoges;
  \institution{CNRS, XLIM UMR 7252}
  \city{Limoges}
  \country{France}  
  \postcode{87060}  
}

\author{Tristan Vaccon}
\orcid{0000-0003-4208-8349}
\affiliation{Universit\'e de Limoges;
  \institution{CNRS, XLIM UMR 7252}
  \city{Limoges}
  \country{France}  
  \postcode{87060}  
}

\renewcommand{\shortauthors}{C. Cornou, S. Naldi and T. Vaccon}

\begin{abstract}
  Polyhedra and spectrahedra over the real numbers, or more generally
  their images under linear maps, are respectively the feasible sets
  of linear and semidefinite
  programming, and form the family of semidefinite-representable sets.
  This paper studies analogues of these sets, as well as
  the associated optimization problems, when the data are taken over a valued 
  field $K$.
  For $K$-polyhedra and linear programming over $K$ we present an algorithm
  based on the computation of Smith normal forms.
  We prove that fundamental properties of semidefinite-representable sets extend
  to the valued setting.
  In particular, we exhibit examples of non-polyhedral $K$-spectrahedra, as well as
 sets that are semidefinite-representable over $K$ but are not $K$-spectrahedra.
\end{abstract}




\maketitle

\section{Introduction}
\label{sec:intro}

Real polyhedra and linear programming (LP) lie at the core
of convex geometry and optimization, as a broad range
of optimization problems can be formulated as, or relaxed to, linear programs.
A non-trivial generalization of LP is semidefinite programming (SDP),
which consists in minimizing a linear function over the set of
positive semidefinite affine combinations of a finite collection of
real symmetric matrices.
This problem has attracted considerable attention in the
communities of real algebra and polynomial optimization,
primarily due to the relaxation method known as the moment-SOS hierarchy
\cite{henrion2020moment}.

Several properties of polyhedra extend to spectrahedra, the feasible
sets of SDP, {\it e.g.}, spectrahedra are convex basic semialgebraic sets,
defined by a positivity condition on a pencil of real symmetric matrices,
known as linear matrix inequality (LMI).
Nevertheless spectrahedra are not closed under linear images; that is,
projections of spectrahedra --- referred to as semidefinite-representable sets ---
are in general not spectrahedra. Nemirovski
\cite{nemirovski2007advances} raised the question of whether every convex semialgebraic set admits
a semidefinite representation, that is, whether it can be realized as the projection
of a spectrahedron. This question was later conjectured affirmatively \cite{helton_sufficient_2008} 
but was subsequently disproved by Scheiderer, who
exhibited compact counterexamples \cite{scheiderer_spectrahedral_2017}.

A further distinction from polyhedra is that the base field over which a LMI is defined may be insufficient to describe the points of the corresponding spectrahedron. In particular, spectrahedra defined over $\QQ$ may fail to admit rational points; consequently, the solution of a semidefinite program defined over $\QQ$ need not be expressible as a rational combination of the input data. Nevertheless, such solutions are algebraic over the base field. Therefore, assuming that the input data of a semidefinite program belong to an effective field $K$, it follows that a solution can be described and computed as an element of a vector space over a suitable algebraic extension of $K$. This observation gives rise to natural algorithmic questions. Complexity bounds for the feasibility problem in semidefinite programming were obtained in \cite{porkolab1997complexity} and rely on general bounds from quantifier elimination \cite{renegar1992computational}. More recently, explicit exact algorithms exploiting the underlying determinantal structure of semidefinite programs have been developed \cite{henrion2016exact,NALDI2018206}.

Nevertheless the complexity status of SDP, unlike that of LP, remains largely
open. This lack of understanding has spurred the development of variants of the problem beyond the classical 
real setting.
In this direction, tropical convexity --- and more precisely tropical analogues of polyhedra and spectrahedra ---
were introduced in \cite{develin2004tropical,allamigeon_tropical_2020} and tropical techniques have since been applied to
non-Archimedean SDP \cite{allamigeon2016solving}. Notably, tropical spectrahedra are known to satisfy
the Helton-Nie property \cite{allamigeon2019tropical}. More recently, linear matrix inequalities with parametric coefficients
have been investigated in \cite{naldi2025solving}.

\paragraph{Main contributions.}
In this paper, we consider the setting in which the input data are defined over $(K,\val)$, where $K$ is a field and
$\val$ is a valuation on $K$. For simplicity, we assume throughout that $\val$ is discrete and that $(K,\val)$ is complete, although some results do not require this assumption. The prototype of valued field we target is the field $K=\QQ_p$ of $p$-adic numbers, equipped with its $p$-adic valuation.


A polyhedron over a valued field is naturally defined as the set of vectors for which certain
linear forms have nonnegative or infinite valuation (see \Cref{def_polyhedra}). Our main result in this
context is the following theorem, proved in \Cref{ssec:projections}:

\begin{theorem}[direct image (DI)] \label{theo:DI}
Let $f : K^n \to K^m$ be an affine map and $\PP$ a polyhedron in $K^n.$
Then $f(\PP)$ is a polyhedron in $K^m$.
\end{theorem}

Next we turn our attention to linear programming. In \Cref{sec:LP}, we present
an algorithm for solving LP problems with data over $(K,\val)$, based on the computation of 
Smith Normal Forms, and we prove its correctness (see \Cref{algo_LP}).

In \Cref{sec:spectrahedra}, we introduce spectrahedra and semidefinite-re\-pre\-sen\-ta\-ble sets
over $(K,\val)$. A spectrahedron over a valued field is defined as the set of matrices
in an affine space that are positive semidefinite; that is, all of whose eigenvalues
have nonnegative valuation in some algebraic closure of $K$. As in the real case, this gives rise
to a class of semialgebraic sets that includes polyhedra.

Nevertheless, in contrast with the real case, we identify a class of subsets of $K$
(that is, in dimension one) that are semidefinite-representable but not spectrahedral, 
under mild assumptions:

\begin{theorem}\label{main_2}
Annuli of $K$ are semidefinite-representable but if the residue field of $K$ is infinite, no (non-trivial) annulus of $K$ can be a spectrahedron of $K.$
\end{theorem}

\section{Preliminaries}
\label{sec:prelim}

Throughout this paper, $(K,\val)$ denotes a complete valued field; that is, $K$ is a field,
$\val : K \twoheadrightarrow \Gamma \cup \{+\infty\}$ is a valuation on $K$, where $\Gamma$ 
is an ordered abelian group, and $K$ is complete with respect to the induced metric.
The valuation is often assumed to be discrete ($\Gamma=\ZZ$)
with uniformizer $\pi \in K$. The valuation ring of $K$ is defined as
$$
\OK = \{x \in K : \val(x) \geq 0\},
$$
with residue field $\kappa = \OK/\mathfrak{m}_K$, where
$\mathfrak{m}_K = \{x \in K : \val(x) > 0\}$.
An element of $\OK$ is called integral (at $\val$).
For $k \in \N$, we write $O(\pi^k)$ for $\pi^k \OK$.
Typical examples of such fields include the $p$-adic numbers 
$\QQ_p$, equipped with the $p$-adic valuation, in which case $\OK = \ZZ_p$, 
and the field of Laurent series
$K=\QQ(\!(t)\!)$ with the $t$-adic valuation,
for which $\OK = \QQ \llbracket t \rrbracket$.
For general background on valued fields, see \cite{Serre:1979,engler2005valued},
and for effective computations over $p$-adic numbers, see \cite{caruso_computations_2017}.

The $K$-vector space of polynomials of degree $\leq d$ on variables
$x=(x_1,\ldots,x_n)$
with coefficients in $K$ is denoted by $K[x]_{d}$. The ring of matrices of
size $p \times q$ with entries in a ring $R$ is $R^{p \times q}$ and the general
linear group in $R^{n \times n}$ is $\GL_n(R)$.
For the sake of ease of writing, block-diagonal matrices with blocks
$B_1,\ldots,B_d$ are denoted by $\diag(B_1, \ldots, B_d)$, the null matrix of
size $p \times q$ by ${\bf 0}_{p,q}$, and the $p \times p$ identity matrix by $I_p$.

We recall the definition of Smith Normal Form, in our context.

\begin{definition}[Smith Normal Form]\label{smith_nf}
  Assume $(K,\val)$ is a discrete valued field with uniformizer $\pi$.
  Let $M \in K^{m \times n}$ be of rank $r$. There exist unique integers
  $a_1 \leq \cdots \leq a_r$, and
  matrices $P \in \GL_n(\OK)$ and $Q \in \GL_m(\OK)$, such that
  $$
  \mathrm{SNF}(M) \coloneq QMP^{-1} = \diag(\pi^{a_1}, \pi^{a_2}, \ldots, \pi^{a_r}, {\bf 0}_{m-r,n-r}).
  $$
  The matrix $\mathrm{SNF}(M)$ is called the \emph{Smith Normal Form (SNF)} of $M$, and the elements $\pi^{a_i}$ the
  \emph{invariant factors}.
\end{definition}
\begin{remark}
  The valuation of the first invariant factor of $\mathrm{SNF}(M)$ of a matrix
  $M \in K^{n \times n}$ is the minimum of the valuations of the entries of
  $M$. In particular, if $M \in \OK^{n \times n}$, then $\mathrm{SNF}(M) \in \OK^{n \times n}$.
\end{remark}

Let us now briefly recall classical, that is, real polyhedra and spectrahedra;
see \cite{netzer2023geometry} for a recent and comprehensive survey.
Let $\sym_d(\R) \subset \R^{d\times d}$ be the vector space of $d \times d$ real symmetric
matrices. A matrix $M \in \sym_d(\R)$ is {positive semidefinite} ($M \succeq 0$) if and only if
 the quadratic form $x \mapsto x^TMx$ is nonnegative for all $x\in \R^n$, that is, precisely when
  the eigenvalues of $M$ are nonnegative. The set $\sym_d^+(\R) \subset \sym_d(\R)$ of positive 
  semidefinite matrices is a closed convex cone with non-empty interior in $\sym_d(\R)$. 
  It is also a semialgebraic set, as $M \succeq 0$ if and only if its principal minors are nonnegative.

Let $A_0,A_1,\ldots,A_n \in \sym_d(\R)$, and let $\calL = A_0+\span{A_1,\ldots,A_n}$ be the affine space
containing $A_0$ with direction the vector space spanned by $A_1,\ldots,A_n$ (which we may assume
to be linearly independent). The intersection
$$
\calL \cap \sym_d^+(\R) = \left\{A(x) \coloneq A_0+\sum_i x_i A_i \in \sym_d(\R) :
A(x) \succeq 0\right\},
$$
or equivalently, its pre-image $S = \{x \in \R^n : A(x) \succeq 0\}$ under the map $x \mapsto A(x)$, is called
a {real spectrahedron}. As an affine section of $\sym_d^+(\R)$ --- or a linear preimage --- real spectrahedra
are closed convex semialgebraic sets, possibly with empty interior.

\begin{example}\label{ell3}
  Let $d=n=3$. The $3$-elliptope is the three-di\-men\-sio\-nal spectrahedron
  $\mathcal{E}_3$ given by
  $$
  \left\{
  x
  \in \R^3 :
  \left[
  \begin{smallmatrix}
    1 & x_1 & x_2 \\
    x_1 & 1 & x_3 \\
    x_2 & x_3 & 1
  \end{smallmatrix}
  \right]
  \succeq 0
  \right\}
  =
  \left\{
  x
  \in \R^3 :
  \begin{array}{l}
    {\footnotesize 3-\sum x_i^2 \geq 0} \\
    {\footnotesize 1+2x_1x_2x_3-\sum x_i^2 \geq 0}
  \end{array}
  \right\}.
  $$
  The polynomials on the right are the (non-constant) coefficients of the
  univariate polynomial $p(t) = \det(A(x)+t I_3)$. Equivalently $\mathcal{E}_3$
  is defined by positivity of the principal minors of the defining matrix:
  $$
  1-x_i^2 \geq 0, \, i=1,2,3, \,\,\, \text{ and } \,\,\,
  1+2x_1x_2x_3-x_1^2-x_2^2-x_3^2 \geq 0
  $$
\end{example}

The optimization problem of miminizing a linear function over a real spectrahedron
is called {semidefinite programming} (SDP).
{Real polyhedra} --- subsets of $\R^n$ defined by finitely many affine inequalities --- are special cases of real
spectrahedra, where the defining matrices $A_0, \ldots, A_n$ can be chosen diagonal, so that
$A = \diag(\ell_1,\ldots, \ell_d)$ for some affine functions $\ell_i$. In other words, SDP generalizes LP.

A more general class of convex semialgebraic sets is obtained via linear projections of real spectrahedra,
called semidefinite-re\-pre\-sen\-ta\-ble sets: these are sets of the form
$$
\left\{x = \left[\begin{smallmatrix} x_1 \\ \vdots \\ x_n \end{smallmatrix}\right] \in \R^n : \exists\,y\in\R^m, \, A_0 + \sum_{i=1}^n x_i A_i + \sum_{j=1}^m y_j B_j \succeq 0\right\}.
$$
When $m>0$, the set is also referred to as a spectrahedral shadow. There exist examples of
spectrahedral shadows that are not spectrahedra (cf. Ex. \ref{fermat_quartic}) as well as compact
convex semialgebraic sets that are not semidefinite-representable \cite{scheiderer_spectrahedral_2017}.

\begin{example}
\label{fermat_quartic}
The basic closed semialgebraic set $\mathcal{R} = \{(x_1,x_2) \in \R^2 : 1-x_1^4-x_2^4 \geq 0\}$ is not a spectrahedron but it is a spectrahedral shadow with semidefinite-representation
$$
\left\{x \in \R^2 :
\exists\, y \in \R^2,
\diag\left(
\left[
\begin{smallmatrix}
  1+y_1 & y_2 \\
  y_2 & 1-y_1
\end{smallmatrix}
\right],
\left[
\begin{smallmatrix}
  1 & x_1 \\
  x_1 & y_1
\end{smallmatrix}
\right],
\left[
\begin{smallmatrix}
  1 & x_2 \\
  x_2 & y_2
\end{smallmatrix}
\right]
\right)
\succeq 0
\right\}
$$
\end{example}

\section{Polyhedra}
\label{sec:polyhedra}

In this section, we define polyhedra over a complete discrete valued field $(K,\val)$
with a uniformizer $\pi$, and we prove that the class of polyhedra is closed under
linear transformations, as in the real case. 

\begin{definition}
  \label{def_polyhedra}
  A \emph{polyhedron} is a subset $\PP \subset K^n$ of the form
  \begin{equation}\label{eq_def_polyhedra}
  \PP =
  \left\{
  x \in K^n :
  \begin{array}{ll}
    \val(\ell_i(x)) \geq 0 & i=1,\ldots,d \\
    \val(m_j(x)) = +\infty & j=1,\ldots,e
  \end{array}
    \right\},
  \end{equation}
  for some $\ell_1,\ldots,\ell_d,m_1,\ldots,m_e \in K[x]_1$.
\end{definition}

Below we use the following notation for $z = (z_1,\ldots,z_d) \in K^d$:
$$
z \succeq 0
\,\,\,\,
\xleftrightarrow{\text{def}}
\,\,\,\,
\forall i=1,\ldots,d, \, \val(z_i)\geq 0.
$$
yielding the following matrix form for $\PP$ (according to notation of
\Cref{def_polyhedra}):
\begin{equation}\label{eq_def_polyhedra_compact}
  \PP = \left\lbrace x \in K^n : Ax+v \succeq 0, Bx+w = 0 \right\rbrace,
\end{equation}
for $A \in K^{d \times n}$, $v \in K^d$, $B \in K^{e \times n}$, and
$w \in K^e$ such that $Ax+v = (\ell_1,\ldots,\ell_d)^T$
and $Bx+w = (m_1,\ldots,m_e)^T$.
Remark that \Cref{def_polyhedra} includes the case $d=0$ corresponding to affine spaces, as in the real case.
Moreover by definition the class of polyhedra is closed under intersection.

\begin{example}
  The {closed unit ball} $\OK^n=B_\infty(K^n)$ in $K^n$ for the supremum norm $||\cdot||_{\infty}$ is the
  polyhedron defined as the set of $x=(x_1,\ldots,x_n) \in K^n$ such that $||x||_{\infty} \coloneq
  \max_i\{\vert x_i \vert\} \leq 1$, in other words, such that $\val(x_i) \geq 0$, for all $i=1, \ldots, n$.
  As such, $\OK^n$ is the polyhedron in $K^n$ defined by the affine polynomials
  $\ell_i(x)=x_i$, $i=1,\ldots, n$, and $e=0$, that is $\OK^n = \{x \in K^n : x \succeq 0\}$.
\end{example}

\subsection{Projections and direct image of polyhedra}
\label{ssec:projections}
The goal of this section is to prove \Cref{theo:DI}, that is, that
the class of polyhedra is stable under application of affine maps. This is done
by reducing to particular cases. The strategy to deduce the main theorem DI 
is given in \Cref{fig:di-properties}.

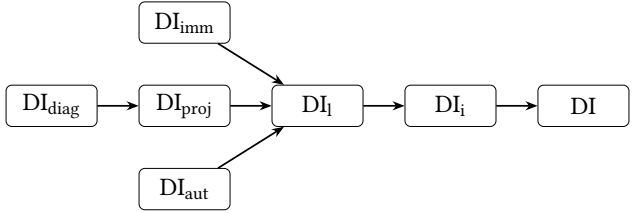
\begin{figure}[!ht] 
\centering
\scalebox{0.55}{
\begin{tikzpicture}[
  node distance=1cm and 1cm,
  box/.style={draw, rounded corners, minimum width=2.2cm, minimum height=1cm, align=center, font=\Huge},
  thickarrow/.style={->, >=Stealth, line width=1.2pt}
]
\node[box] (DIdiag) {DI${}_{\mathrm{diag}}$};
\node[box, right=of DIdiag] (DIproj) {DI${}_{\mathrm{proj}}$};
\node[box, right=of DIproj] (DIl) {DI${}_{\mathrm{l}}$};
\node[box, right=of DIl] (DIi) {DI${}_{\mathrm{i}}$};
\node[box, right=of DIi] (DI) {DI};
\node[box, above=of DIproj] (DIimm) {DI${}_{\mathrm{imm}}$};
\node[box, below=of DIproj] (DIaut) {DI${}_{\mathrm{aut}}$};
\draw[thickarrow] (DIdiag) -- (DIproj);
\draw[thickarrow] (DIproj) -- (DIl);
\draw[thickarrow] (DIl) -- (DIi);
\draw[thickarrow] (DIi) -- (DI);
\draw[thickarrow] (DIimm) -- (DIl);
\draw[thickarrow] (DIaut) -- (DIl);
\end{tikzpicture}
}
\caption{Strategy of proof of Theorem \ref{theo:DI}.}
\label{fig:di-properties}
\end{figure}

We begin with the case of linear maps in $\GL_n(\OK)$.
We often fix a linear basis $\xi_1,\ldots,\xi_n$ of $K^n$ which we call canonical.

\begin{lemma}[DI${}_{\mathrm{aut}}$]
Let $f \in \GL_n(\OK)$ and $\PP$ a polyhedron in $K^n$.
Then $f(\PP)$ is a polyhedron. \label{lem:DIaut}
If in addition, $\PP$ is defined only by inequalities,
then so is $f(\PP)$.
\end{lemma}
\begin{proof}
  Assume $\PP$ is defined in matrix form as in \eqref{eq_def_polyhedra_compact},
  and let $U$ be the matrix of
  $f$ in the canonical basis of $K^n$. Let
  \[\Qc=\left\lbrace y \in K^n,  AU^{-1} y+v \succeq 0, BU^{-1} y+w=0 \right\rbrace.\]
  Then it follow that $\Qc$ is a polyhedron and $f(\PP)=\Qc.$
  It is also clear that if $\PP$ is defined only by inequalities,
  then so is $f(\PP)$.
\end{proof}

If $f$ is diagonal, we can also conclude.

\begin{lemma}[DI${}_{\mathrm{diag}}$]
  Let $f : K^n \to K^n$ be a linear map and $\PP$ a polyhedron in $K^n$ defined only by inequalities.
  If the matrix of $f$ in the canonical basis of $K^n$ is diagonal,
  then $f(\PP)$ is a polyhedron. \label{lem:DIdiag}
\end{lemma}
\begin{proof}
As in Lemma \ref{lem:DIaut}, a change of variables suffices to conclude.
\end{proof}

We assume now $f$ is a projection erasing the last coordinate.

\begin{lemma}[DI${}_{\mathrm{proj}}$]
  Let $f : K^n \to K^{n-1}, \: (x_1,\dots,x_{n-1},x_n) \mapsto (x_1,\dots,x_{n-1})$, and let
  $\PP$ be a polyhedron in $K^n$ defined only by inequalities. 
  Then $f(\PP)$ is a polyhedron defined only by inequalities. \label{lem:DIproj}
\end{lemma}
\begin{proof}
  We assume that $\PP \neq \emptyset$ and that it is given in
  matrix form: $\PP = \{x \in K^n : Ax+v \succeq 0\}$, for some 
  $A \in K^{d \times n}$, $v \in K^d$,
  and $x=(x_1,\dots,x_n)$ a column vector. 
  
  We can compute a Smith Normal Form SNF($A'$)
  of the $d \times (n-1)$ matrix $A'$ defined by the 
  $n-1$ first columns of $A$
  \[
\mathrm{SNF}(A') = \diag(\delta_1,\ldots,\delta_r,{\bf 0}_{d-r,n-1-r}) \in K^{d \times (n-1)}
\]
to obtain a decomposition of the form
\[
Q_A A P_A^{-1} =
\begin{bmatrix}
 \mathrm{SNF}(A') & a
\end{bmatrix}
=
\left[
\begin{smallmatrix}
\delta_1 &        &          &   &        &   & a_1     \\[-.25cm]
         & \ddots &          &   &        &   &         \\[-.35cm]
         &	  & \delta_r &   &        &   &  \vdots \\
         &	  &          & 0 &        &   &         \\[-.25cm]
         &	  &          &   & \ddots &   &         \\
         &	  &          &   &	  & 0 & a_{n-1} \\ 
         &	  &          &   &	  &   & a_n 	\\
         &	  &          &   &	  &   & \vdots 	\\
         &	  &          &   &	  &   & a_d 	
\end{smallmatrix}
\right]
\]
with $Q_A \in \GL_d(\OK)$ and $P_A \in \GL_n(\OK)$ such that $P_A \xi_n = \xi_n$,
and $a = (a_1,\ldots,a_d)^\intercal$.

Multiplying on the left by $Q_A^{\pm 1}$ does not
modify whether the valuation inequalities are satisfied.
Moreover, thanks to Lemma \ref{lem:DIaut},
$\PP$ can be written as
$\PP = P_A^{-1} \PP'$ for some polyhedron $\PP' \subset K^n$
and we are reduced to proving that
$f(\PP')$ is a polyhedron for $\PP'$
defined by the following linear matrix inequality (for some $v' \in K^d$):
\[
Q_A A P_A^{-1}
\begin{bmatrix}
  x_1 \\ \vdots \\ x_n
\end{bmatrix}
+ v' \succeq 0.
\]
We remark that 
$a_s x_n + v'_s \succeq 0$
if and only if $x_n = -a_s^{-1} v'_s + O(\pi^{- \val(a_s)}).$
This exactly corresponds to $x_n \in B(-a_s^{-1} v'_s,\val(a_s)).$
Since the conditions have to be compatible (we have assumed
$\PP\neq\emptyset$), and thanks to ultrametricity,  
the intersection of these conditions for $s \in \llbracket r+1, d \rrbracket$
is then given by the smallest of these balls.
Let us assume it is the one defined by the last inequality, \textit{i.e.} for $s = d.$
The only condition on $x_n$
is then $\val(a_dx_n + v'_d) \geq 0.$

Thanks to Lemma \ref{lem:DIdiag},
we can further assume that the $\delta_i$'s and $a_n$ are all $1$'s
and thus with abuse of notation we assume $\PP'$ is defined by
the linear matrix inequality:
\[
\left[
\begin{smallmatrix}
1 &        &   &   &        &   & a_1     \\[-.25cm]
  & \ddots &   &   &        &   &         \\[-.35cm]
  &	   & 1 &   &        &   &  \vdots  \\
&	   &   & 0 &        &   &    \\[-.25cm]
  &	   &   &   & \ddots &   &         \\
  &	   &   &   &	    & 0 & a_{d-1} \\ 
  &	   &   &   &	    &	& 1 	  \\
\end{smallmatrix}
\right]
\begin{bmatrix}
  x_1 \\
  \vdots \\
  x_n
\end{bmatrix}
+
\begin{bmatrix}
  v'_1 \\
  \vdots \\
  v'_d
\end{bmatrix}
\succeq 0.
\]

Finally, we can also assume that, up to a permutation of the variables, $\val(a_1)= \min_{i=1}^{r} \val(a_i).$
We distinguish two cases: $\val(a_1) \geq 0$ and $\val(a_1)<0.$
Let us first assume that $\val(a_1) \geq 0$.

Let $x \in \PP'$.
Then $x_n=-v'_d+O(1)$
and $x_1+a_1 x_n+v'_1=O(1).$
Consequently, $x_1-a_1 v'_d +b_1=O(1)+O(\pi^{\val(a_1)})=O(1).$
Similarly, for all $i \in \llbracket 2,r \rrbracket,$
$x_i-a_i v'_d+v'_i=O(1)+O(\pi^{\val(a_i)})=O(1).$

We define the polyhedron $\Qc \subset K^{n-1}$ by the following inequalities:
 for all $i \in \llbracket 1,r \rrbracket,$
$\val(x_i-a_iv'_d+v'_i) \geq 0$. Then the previous computations prove that $f(\PP')\subset \Qc.$

Conversely, let $(x_1,\dots,x_{n-1}) \in \Qc.$
Then let $x_n \coloneq -v'_d.$
We check that for $i \in \llbracket 1,r \rrbracket,$
$x_i+a_ix_n+v'_i=x_i-a_iv'_d+v'_i=O(1).$
In addition $x_n+v'_d=O(1)$ is satisfied
and as we have assumed it is the coarsest of the inequalities
for $s \in \llbracket r+1,d \rrbracket,$, it implies that 
for all $s \in \llbracket r+1,d-1 \rrbracket,$
$a_s x_n+v'_s=O(1)$.
Hence $(x_1,\dots,x_n) \in \PP'$
which means that $(x_1,\dots,x_{n-1}) \in f(\PP').$
Thus $f(\PP')= \Qc$, that is, $f(\PP')$ is a polyhedron.

We now deal with the second case, $\val(a_1)<0.$
Let $x \in \PP'$.
Then $x_n=-v'_d+O(1)$
and $x_1+a_1 x_n+v'_1=O(1).$
Applying the first equality in the second, we get that
$x_1-a_1 v'_d +v'_1=O(1)+O(\pi^{\val(a_1)})=O(\pi^{\val(a_1)}).$
Thus $a_1^{-1}x_1- v'_d +a_1^{-1}v'_1=O(1)$.
In addition, $x_n=-a_1^{-1}(x_1+v'_1)+O(\pi^{-\val(a_1)}).$
For $i \in \llbracket 2,r \rrbracket,$ we plug this equality for $x_n$ in 
$x_i+a_i x_n+v'_i=O(1)$ to get
$$
x_i-a_1^{-1}a_i(x_1+v'_1)+v'_i=O(1)+O(\pi^{\val(a_i)-\val(a_1)})=O(1)
$$
using that $\val(a_i) \geq \val (a_1).$

We define the polyhedron $\Qc \subset K^{n-1}$ by the following inequalities:
$\val(a_1^{-1}x_1- v'_d +a_1^{-1}v'_1) \geq 0$ and for all $i \in \llbracket 2,r \rrbracket,$
$\val(x_i-a_1^{-1}a_i(x_1+v'_1)+v'_i) \geq 0.$
Then the previous computations prove that $f(\PP')\subset \Qc.$

Conversely, let $(x_1,\dots,x_{n-1}) \in \Qc.$
Let $x_n \coloneq -a_1^{-1}(x_1+v'_1).$
Then
$x_1+a_1 x_n+v'_1=0=O(1).$
Let $i \in \llbracket 2,r \rrbracket,$
then $x_i+a_i x_n+v'_i=x_i-a_1^{-1}a_i(x_1+v'_1)+v'_i=O(1)$
since $(x_1,\dots,x_{n-1}) \in \Qc.$
Furthermore, 
$x_n+v'_d=-a_1^{-1}(x_1+v'_1)+v'_d=O(1)$ thanks to the inequality satisfied by $x_1.$
Since the inequality for $s=d$ has been assumed to be 
the coarsest of the inequalities for $s \in \llbracket r+1,d \rrbracket,$
we also get that that for $s \in \llbracket r+1,d-1 \rrbracket,$
$a_s x_n+v'_s = O(1)$.
Hence $(x_1,\dots,x_n) \in \PP'$
which means that $(x_1,\dots,x_{n-1}) \in f(\PP').$
Thus $f(\PP')= \Qc$ and $f(\PP')$ is a polyhedron.

Therefore, in all cases, $f(\PP)$ is a polyhedron, and inequalities are enough to define
$f(\PP)$ as a polyhedron.
\end{proof}

\begin{lemma}[DI${}_{\mathrm{imm}}$]
  Let $f : K^l \to K^n$ be the immersion defined by
  $(x_1,\dots, x_l) \mapsto (x_1,\dots, x_l,0,\dots, 0)$ and let $\PP$ be
  a polyhedron in $K^l$. Then $f(\PP)$ is a polyhedron. \label{lem:DIimm}
\end{lemma}
\begin{proof}
Adding the equalities $x_{l+1}=0, \dots, x_n=0$ to the equations and inequations in
$x_1,\dots, x_l$ defining $\PP$, the result is clear.
\end{proof}

We can now generalize to linear mappings and polyhedra defined only by inequalities.

\begin{proposition}[DI${}_{\mathrm{l}}$]
Let $f : K^n \to K^m$ be a linear map and $\PP$ a polyhedron in $K^n$ defined only by inequalities.
Then $f(\PP)$ is a polyhedron. \label{prop:DIl}
\end{proposition}
\begin{proof}
  We identify $f$ with its matrix in $K^{m \times n}$, and we
  assume that $\emptyset \neq \PP = \{x \in K^n : \val(l_i(x)) \geq 0, i=1,\ldots,d\}$.
  
  By Definition \ref{smith_nf}, the matrix of $f$ can be decomposed
  $f = Q^{-1} \Delta P$ with $P \in \GL_n(\OK)$, $Q \in \GL_m(\OK)$ and
  $$
  \Delta = \diag(\delta_1, \ldots, \delta_r, {\bf 0}_{m-r,n-r}) = \mathrm{SNF}(f)\in K^{m \times n}
  $$

By Lemma \ref{lem:DIaut}, $\Qc := P \PP$ is a polyhedron defined only by inequalities.
We now prove that $\Delta \Qc$ is a polyhedron. One can write that $\Delta$ is the composition
of the immersion 
$K^r \rightarrow K^n$ sending
$(x_1,\dots,x_r)\mapsto (x_1,\dots,x_r,0,\dots,0),$ 
 the diagonal invertible mapping
$K^r \rightarrow K^r$ sending
$(x_1,\dots,x_r)\mapsto (\delta_1 x_1,\dots,\delta_r x_r)$,
and the projections
\begin{equation*}
  \begin{aligned}
    (x_1,\dots,x_{r+1}) &\mapsto (x_1,\dots,x_{r}),   \\
    & \,\,\,\, \vdots \\
    (x_1,\dots,x_{n}) &\mapsto (x_1,\dots,x_{n-1}).
  \end{aligned}
\end{equation*}
Thus, thanks to Lemmas \ref{lem:DIimm}, \ref{lem:DIproj} and \ref{lem:DIaut},
$\Delta \Qc$ is a polyhedron. One can note that it is only when applying the
immersion as the last mapping in the composition chain that equalities are
needed to define the last polyhedron.

Finally, thanks to Lemma \ref{lem:DIaut},
$Q^{-1} \Delta \Qc = f(\PP)$ is a polyhedron.
This concludes the proof.
\end{proof}

We can generalize from linear mappings to affine mappings.
\begin{proposition}[DI${}_{\mathrm{i}}$]
Let $f : K^n \to K^m$ be an affine map and $\PP$ a polyhedron in $K^n$ defined only by inequalities.
Then $f(\PP)$ is a polyhedron. \label{prop:DIi}
\end{proposition}
\begin{proof}
  Let $v \in K^m$. One can see by immediate change
  of variables that $\Qc$ is a polyhedron in $K^m$ if and only if
  $\Qc+v$ is a polyhedron in $K^m$.
  Consequently, translating by $-f(O)$ in the codomain, where
  $O \in K^n$ is the origin, one can reduce to \Cref{prop:DIl}.
\end{proof}

We conclude with the proof of the main result of this section.

\begin{proof}[Proof of \Cref{theo:DI}]
Recall that here $f : K^n \to K^m$ is an affine map and $\PP \subset K^n$ is a polyhedron.
Let us write $\PP = \PP_I \cap \PP_E$
with $\PP_I \coloneq \left\lbrace x \in K^n,  Ax+v \succeq 0 \right\rbrace$
and $\PP_E \coloneq \left\lbrace x \in K^n,  Bx+w=0 \right\rbrace$.
Then $f(\PP) = f_{\PP_E}(\PP_I)$, {\it i.e.}, the image of $\PP_I$ through the restriction
map of $f$ to $\PP_E$.

Up to a translation, we can assume that $\PP_E$ is a linear vector space:
we are reduced to proving that, if $\tilde{f}: \PP_E \to K^m$ is an affine map
and $\PP_I$ is a polyhedron in the vector space $\PP_E$ defined only by inequalities,
then $\tilde{f}(\PP_I)$ is a polyhedron. This is the case thanks to Proposition \ref{prop:DIi}
and the theorem is proven.
\end{proof}

We recall that the Minkowski sum of two sets $S$ and $T$ is $S+T = \{s+t : s \in S, t \in T\}$.

\begin{corollary}
The Minkowski sum of two polyhedra in $K^n$ is a polyhedron in $K^n$.
\end{corollary}
\begin{proof}
  Straightforward from the equality
  $\PP_1 + \PP_2 = \{z \in K^n : x \in \PP_1, y \in \PP_2, x+y=z\}$
  and applying \Cref{theo:DI}.
\end{proof}

\begin{proposition}[Characterization of polyhedra] \label{prop:carac}
    A non-empty polyhedron is the image of a polydisc under an affine map.
\end{proposition}
\begin{proof}
    Let $\PP = \{x \in K^n : Ax+v \succeq 0, Bx + w = 0\}$ be a polyhedron in $K^n$.
    If $\PP \neq \emptyset$, then there exists $x_0 \in K^n$ such that $Bx_0+w=0$.
    Let $J \in K^{n \times R}$ such that $\mathrm{Im}\,J = \ker\,B$ where $R$ denotes
    the dimension of $\ker\,B$. Then $\PP = x_0 + J\PP'$ where
    $A' \coloneq AJ$, $v' \coloneq Ax_0 + v$ and
    $\PP' \coloneq \{y \in K^R : A'y + v' \succeq 0 \}$.

    Now let $S \in K^{n\times R}$ be the SNF of $A'$ and $Q \in \GL_n\left(\mathcal{O}_K\right)$
    and $P \in \GL_R(\mathcal{O}_K)$ such that $A'= Q S P^{-1}$.
    Then
    $$
    \PP' = \{y \in K^R : QSP^{-1}y + v' \succeq 0\} = P \mathcal{Q},
    $$
    with $\mathcal{Q} := \{z \in K^R : Sz + Q^{-1}v' \succeq 0 \}$.
    Let $v'' \coloneq Q^{-1}v'$ then the set
    $\mathcal{Q} = \{z \in K^R : s_i z_i + v''_i \succeq 0, i= 1,\ldots,R\}$
    is either empty or a polydisc (possibly with infinite radius).
    Thus $\PP = x_0+JP\mathcal{Q}$ is empty or the image of a polydisc by an affine map.
\end{proof}

\begin{corollary} \label{cor:caracdim1}
   Polyhedra in $K$ are either the empty set or balls, possibly with infinite radius.
\end{corollary}

\begin{proof}
    Let $\PP \subset K$ be a non-empty polyhedron. By \Cref{prop:carac} there exists a polydisc $D = \{x \in K^n : \val(s_i x_i - b_i) \ge 0, i=1,\ldots,n \}$ for some $s_1,\ldots,s_n, b_1,\ldots,b_n \in K ^n$ and an affine map $f : K^n \to K$ such that $\PP = f(D)$. Let $a_0,a_1,\ldots,a_n \in K$ such that $f(x) = a_0 + a_1 x_1 + \ldots + a_n x_n$. We can assume, without loss of generality, that the $a_i$'s are nonzero. 

    If there exists $i_0 \in \{1,\ldots,n\}$ such that $s_{i_0} = 0$ then $\PP = f(D) = K$ and $\PP$ is a ball of infinite radius. Otherwise, we have
    $$
    D = \left\{x \in K^n : x_i = \frac{b_i}{s_i} + O\left( \pi^{\val(b_i/s_i)} \right), i=1,\ldots,n\right\}
    $$
    and therefore for all $x \in K$, one has $x \in P$ if and only if
    $$
    x = a_0 + \sum_{i=1}^{n}\left(\frac{a_i b_i}{s_i} + O( \pi^{\val(a_i b_i/s_i)}) \right)
    = a_0  + \sum_{i=1}^{n} \frac{a_i b_i}{s_i} + O(\pi^d)
    $$
    where $d = \min \limits_{i = 1,\ldots,n} \val(a_i b_i/s_i)$. Thus
    $\PP$ is the ball described by
    $$
    \PP = \left\{x \in K : \val\left(x - a_0 - \sum_{i=1}^{n} \frac{a_i b_i}{s_i}\right) \ge d\right\}.
    $$
\end{proof}


\section{Linear programming}
\label{sec:LP}
A {linear programming problem} (or {linear program})
{with data $(A,b,c,\allowbreak D,e)$} is defined as:

\begin{equation}
  \tag{LP}\label{LP}
\begin{array}{rcll}
  p^* & \coloneq & \inf_{x \in K^n} & \val \left\langle c, x \right\rangle \\
  &    & \text{s.t.}         & A x + b \succeq 0\\
  & & & D x = e
\end{array}
\end{equation}
for a vector $c \in K^n$ defining a cost function $\val\left\langle c, x \right\rangle$,
with $\left\langle c, x \right\rangle \coloneq c_1 x_1+\cdots +c_n x_n$,
$A \in K^{d \times n}$, $b \in K^d$, $D \in K^{m\times n}$, and $e \in K^m$.
When $d=n$ and $A$ is diagonal, 
we refer to \eqref{LP} as a {diagonal} linear program.
We denote the {feasible set} of \eqref{LP} by
\begin{equation*}
\begin{aligned}
  \PP &= \{x \in K^n : Ax + b \succeq 0, Dx = e\}.
\end{aligned}
\end{equation*}
Note that $\PP$ is a polyhedron.
A vector $x \in K^n$ is called {feasible} for~\eqref{LP} if $x \in \PP$. A feasible vector $x^* \in \PP$ is called a {solution} to \eqref{LP} if $p^* = \val\left\langle c,x^*\right\rangle$, {\it i.e.} if the infimum in \eqref{LP} is attained at $x^*$.

\begin{remark}
  Assume $\val$ is a discrete valuation. Then
  if \eqref{LP} admits feasible vectors,
  $p^*$ is either $-\infty$ or attained at a solution $x^*$.
\end{remark}

In this section we present \Cref{algo_LP}, \textbf{SOLVELPval}, for solving problem \eqref{LP}.
The algorithm starts by eliminating the equality constraint defined by $D$ and $e$
by variable substitution. Then the matrix $A$ is replaced by its Smith Normal Form,
and the general linear program is consequently reduced to a diagonal linear program.
The original \eqref{LP} is then easily solved by looking at valuations of input data
in the reduced problem.

\begin{algorithm}[!ht]
    \caption{\textbf{SOLVELPval}($A,b,c,D,e$)}\label{algo_LP}
    \begin{algorithmic}[1]
      \REQUIRE
      $A \in K^{d\times n}$, $b \in K^d$, $c \in K^n$, $D \in K^{m \times n}$ and $e \in K^m$
      \ENSURE [TYPE, $x$], where:
      \begin{itemize}
      \item[]
        TYPE $=$ INFEAS and $x=[]$ if and only if \eqref{LP} is infeasible
      \item[]
        TYPE $=$ UNBOUND and $x=[]$ if and only if \eqref{LP} is feasible and unbounded
      \item[]
        TYPE $=$ FEAS if and only if \eqref{LP} is feasible and has a solution $x$
      \end{itemize}
      \hrule
      \vspace{0.1cm}
      \IF{$Dx = e$ has no solution}
      \RETURN [INFEAS, []] \label{step:linear_empty}
      \ENDIF
      \IF{$D \neq {\bf 0}$} \label{step:reducD:begin}
      \STATE find $J \in K^{n \times R}$, with $R = \dim\ker D$, such that $\text{Im}\, J = \ker D$
      \STATE find $x_0$ such that $Dx_0=e$ \label{step:reducD:end}
      \RETURN $\textbf{SOLVELPval}(AJ,b+A x_0,J^T c, {\bf 0},0)$ \label{step:reducD} \hfill {\color{olive} \# \Cref{reducD}}
      \ENDIF
      \IF{$A$ not in Smith Normal Form} \label{step:solsLP:begin}
      \STATE $[Q,S,P] \leftarrow \textbf{SNF}(A)$
      \STATE $[\mathrm{TYPE}, y] \leftarrow \textbf{SOLVELPval}(S,Qb,(P^{-1})^Tc,{\bf 0},0)$ \label{step:solsLP:end}
      \RETURN $[\mathrm{TYPE}, P^{-1}y]$ \hfill \color{olive} \# \Cref{solsLP} \label{step:solsLP}
      \ENDIF
      \STATE $r \leftarrow \text{rank}(A)$
      \hfill {\color{olive} \# $A$ is in $\mathrm{SNF}$, with diagonal elements $s_{i}$} \label{step:core:begin}
      \IF{$\bigwedge_{i=r+1}^d (\val(b_{i}) \geq 0)$} \label{step:condition}
      \IF{$\bigvee_{i=r+1}^n (c_i \neq 0)$}
      \RETURN [UNBOUND, []] \hfill {\color{olive} \# \Cref{prop:reduc}, item \ref{prop:reduc_it2}} \label{step:unbound}
      \ELSE 
      \STATE $\lambda \leftarrow -\sum_{i=1}^r c_i b_i / s_{i}$
      \STATE $v \leftarrow \min_{1 \leq i \leq r} \val (c_i/s_{i})$
      \IF{$\val \lambda < v$}
      \STATE choose $y \in K^n$ s.\,t. $y_i \in (-b_i+\OK)/s_{i}, i=1,\ldots,r$
      \RETURN [FEAS, $y$] \hfill {\color{olive} \# \Cref{solutions}, item \ref{solutions:it1}}
      \ELSE
      \STATE $j \leftarrow \min_{1 \leq i \leq r} \{i : v = \val (c_i/s_{i})\}$
      \STATE $\xi_j = (0,\ldots,0,1,0,\ldots,0)$, with $1$ at coordinate $j$
      \RETURN [FEAS, $\xi_j$] \hfill {\color{olive} \# \Cref{solutions}, item \ref{solutions:it2}}
      \ENDIF
      \ENDIF 
      \ENDIF 
      \RETURN [INFEAS, []] \label{step:core:end}
    \end{algorithmic}
\end{algorithm}

In order to prove correctness of \Cref{algo_LP}, we proceed step-by-step.
When the system $Dx=e$ has no solution, \eqref{LP} is infeasible and
thus from now on we only consider the case when $Dx=e$ has a solution.

We first prove a reduction result, allowing to consider only iterations
of~\eqref{LP} with data of the form $(A,b,c,{\bf 0},0)$.

\begin{lemma}\label{reducD}
  Let $A \in K^{d \times n}, b \in K^d, c\in K^n$, $D \in K^{m\times n}$,
  $e \in K^m$ and let $J \in K^{n \times R}$ be a matrix such that
  ${\rm Im}\, J = \ker D$ and $R = \dim \ker D$.
  Assume there exists $x_0 \in K^n$ such that $D x_0 = e$, and let
  $x \in \{z \in K^n : Dz = e\} = x_0 + \ker D$.
  Let $y \in K^R$ denote the unique element of $K^n$ such that
  $x = x_0 + J y$. Then:
  \begin{enumerate}
  \item\label{reducD_item1}
    $x$ is feasible for~\eqref{LP} with data $(A,b,c,D,e)$ if and only if 
    $y$ is feasible for~\eqref{LP} with data $(AJ,b+A x_0,J^T c,{\bf 0},0)$.
  \item\label{reducD_item2}
    $x$ is a solution for~\eqref{LP} with data $(A,b,c,D,e)$ if and only if 
    $y$ is a solution for~\eqref{LP} with data $(AJ,b+A x_0,J^T c,{\bf 0},0)$ and
    $\langle c,x^* \rangle = \langle J^T c,y^* \rangle$.
  \end{enumerate}
\end{lemma}
\begin{proof}
  First, if $D={\bf 0}$, then by the assumptions one must have $e=0$, and the problem
  is already reduced (with $J=I_n$ and $x_0=0$).
  Next, assume $D\neq{\bf 0}$, then $Dx=e$ and $Ax + b \le 0$ if and only if
  $AJy + b+Ax_0 \le 0$, proving \Cref{reducD_item1}. Finally, for all
  $c \in K^n$, $\langle c,x\rangle = \langle c, Jy\rangle = \langle J^T c,y\rangle$, which proves \Cref{reducD_item2}.
\end{proof}
 

\Cref{solsLP} allows us to reduce problem \eqref{LP} without linear equations
to diagonal linear programs.

\begin{lemma} \label{solsLP}
  Let $A \in K^{d\times n}$, $b \in K^d$, $c\in K^n$ and let
  $S = \mathrm{SNF}(A) = QAP^{-1} \in K^{d\times n}$, with
  $P \in \GL_n(\OK)$ and $Q\in \GL_d(\OK)$.
  Define $b' = Qb$ and $c' = (P^{-1})^Tc$.
  Then
   \begin{enumerate}
   \item \label{solsLP_item1}
     $x \in K^n$ is feasible for \eqref{LP} with data $(A,b,c,{\bf 0},0)$ if and only if $y = P x \in K^n$ is
     feasible for \eqref{LP} with data $(S,b',c',{\bf 0},0)$.
   \item \label{solsLP_item2}
     $x^* \in K^n$ is a solution to \eqref{LP} with data $(A,b,c,{\bf 0},0)$ if and only if $y^* = P x^*$ is
     a solution to \eqref{LP} with data $(S,b',c',{\bf 0},0)$ and $\langle c,x^* \rangle = \langle c',y^* \rangle$.
   \end{enumerate}
\end{lemma}
\begin{proof}
  Note that the transition matrices $P$ and $Q$ defining the form SNF$(A)$ are invertible in $\OK$
  and as such they preserve inequalities: $Mz \succeq 0$ if and only if $Nz \succeq 0$, with
  $N = \mathrm{SNF}(M)$. Thus one has that $Ax+b = AP^{-1}y +b \succeq 0$ if and only if
  $Q(AP^{-1}y + b) = Sy + b' \succeq 0$; this yields \Cref{solsLP_item1}.
  Moreover, for all $c \in K^n$, $\langle c, x\rangle = \langle c, P^{-1}Px \rangle =
  \langle (P^{-1})^T c, Px \rangle = \langle c',y\rangle$, which proves
  \Cref{solsLP_item2}.
\end{proof}

Finally we show in \Cref{prop:reduc} that solving a diagonal \eqref{LP} amounts
to checking the valuations of input data:

\begin{lemma}\label{prop:reduc}
  Let $(S,b,c) \in K^{d\times n} \times K^d \times K^n$, with $S$ in SNF.
  \begin{enumerate}
  \item \label{prop:reduc_it1}
    \eqref{LP} with data $(S,b,c,{\bf 0},0)$ is feasible if and only if the $d-r$ last coefficients
    of $b$ have nonnegative valuation, for $r = {\rm rank}\,S$.
  \item \label{prop:reduc_it2}
    Assume at least one of the $n-r$ last coefficients
    of $c$ is non-zero. If \eqref{LP} with
    data $(S,b,c)$ is feasible, then it is unbounded and does not admit solutions.
  \end{enumerate}
\end{lemma}
\begin{proof}
  A vector $y = (y_1, \ldots, y_n)^T \in K^n$ is feasible for problem \eqref{LP} with data
  $(S,b,c,{\bf 0}, 0)$ if and only if $\val(s_i y_i + b_i) \geq 0$ for all $i = 1,\ldots,n$.
  However $s_i = 0$ for $n-r \leq i\leq n$. Therefore, if \eqref{LP} with data $(S,b,c,{\bf 0},0)$
  is feasible,
  necessarily the $d-r$ last coefficients of $b$ have nonnegative valuation.
  Conversely, if the $d-r$ last coefficients of $b$ have nonnegative valuation,
  the vector $y \in K^n$ such that
  $y_i = {-b_i}/{s_i}$ if $1 \leq i \leq r$, and $y_i =0$ otherwise, is feasible for \eqref{LP}
  with data $(S,b,c,{\bf 0},0)$.

  Concerning \Cref{prop:reduc_it2}, let $y^*$ be feasible for problem \eqref{LP} with data
  $(S,b,c,{\bf 0},0)$.
  Let $i$ be such that $r+1 \leq i \leq n$ and $c_i \neq 0$. Consider the vector $\xi_i$ whose
  only non-zero
  coefficient is the $i-$th, which is $1$. Then $y^* + \alpha \xi_i$ is feasible for \eqref{LP}
  with data
  $(S,b,c,{\bf 0},0)$, for all $\alpha \in K$ and $\langle c, y^* + \alpha \xi_i\rangle =
  \langle c, y^* \rangle + \alpha c_i$.
  In particular, for all $n \in \mathbb{N}$ if we set $\alpha_n = -({\langle c, y^* \rangle +
    \pi^{-n}})/{c_i}$,
  we have $\val \left(\langle c, y^* + \alpha_n \xi_i \rangle\right) = \val \pi^{-n} = -n$.
  Therefore $p^* = -\infty$ and (LP) is unbounded.
\end{proof}


We conclude by giving explicit solutions to the reduced \eqref{LP}:

\begin{lemma}\label{solutions}
  Let $(S,b,c) \in K^{d\times n} \times K^d \times K^n$, with $S$ in SNF,
  and assume \eqref{LP} with data $(S,b,c,{\bf 0},0)$ is feasible.
  Let $\lambda = -\sum_{i=1}^r c_i b_i/s_i$ and $v = \min_{1\le i\le r} \val (c_i/s_i)$.
  \begin{enumerate}
  \item
    if $\val \lambda < v$, the minimum is $\val \lambda$ and is attained on the whole feasible
    set; \label{solutions:it1}
  \item
    if $\val \lambda \geq v$, the minimum is $v$ and is attained on
    $\xi_j \in K^n$, the vector with all $0$'s except with a $1$ on the $j$-th component,
    with $j$ satisfying $\val (c_j/s_j) = v$. \label{solutions:it2}
  \end{enumerate}
\end{lemma}
\begin{proof}
  The feasible set $\PP' = \{y \in K^n : Sy+b \succeq 0\}$ of the diagonal LP can be written
  as the set of vectors $y \in K^n$ satisfying $s_i y_i + b_i \in \OK$ for $i=1,...,r$.
  In other words:
  $$
  \PP' = \left\{\left[\begin{smallmatrix} y_1\\ \vdots\\ y_n \end{smallmatrix}\right] \in K^n :
  y_i \in -\frac{b_i}{s_i} + \frac{1}{s_i} \OK, 1\le i\le r\right\}.
  $$
  The minimum in \eqref{LP} is the minimum of
  $y \mapsto \val\left(\left\langle c,y \right\rangle\right)$ on $\PP'$.
  The image of $\PP'$ by $y \mapsto \left\langle c,y \right\rangle$ is
  \[
  \left\langle c,\PP' \right\rangle = \sum_{i=1}^r -c_i \frac{b_i}{s_i} +
  \sum_{i=1}^r\left( \frac{c_i}{s_{i}} \OK \right) = \lambda + \pi^{v} \OK.
  \]
  Two cases can occur depending on $v$ and the valuation of $\lambda$.
  If $\val(\lambda) < v$, the minimum of $y\mapsto \val\left(\left\langle c,y \right\rangle\right)$
  over $\PP'$ is reached at any point of $\PP'$, and is equal to $\val(\lambda)$.
  Otherwise, $\lambda \in \pi^{v} \OK$ and the minimum is $v$ and is reached at any points $y$ of $\PP'$ such that
  $\left\langle c,y \right\rangle \in \pi^v\OK$.
\end{proof}

We are now able to state a correctness theorem for \Cref{algo_LP}.
  
\begin{theorem} \Cref{algo_LP} is correct.
\end{theorem}
\begin{proof}
  If the linear system $Dx=e$ has no solution, problem \eqref{LP} is infeasible
  and \Cref{algo_LP} outputs the correct answer at step \ref{step:linear_empty}.
  Else, by \Cref{reducD} the input problem can be reduced to a new linear
  program with $D={\bf 0}$ and $e=0$, which is done at steps
  \ref{step:reducD:begin}-\ref{step:reducD:end},
  and the correct recursion is given at step \ref{step:reducD}.
  The reduction based on SNF from \Cref{solsLP} is performed at
  steps \ref{step:solsLP:begin}-\ref{step:solsLP:end} and the output
  at step \ref{step:solsLP} is correct.
  
  The core of \Cref{algo_LP} is on steps \ref{step:core:begin}-\ref{step:core:end}.
  If $\val(b_i) \geq 0$, for all $i=1,\ldots,r+1$, with $r = \mathrm{rank}(A)$,
  by \Cref{prop:reduc}, item \ref{prop:reduc_it1}, the problem \eqref{LP} is feasible.
  Two cases can occur under this assumption:
  either at least one of the $n-r$ last coefficients of $c'$ is non-zero,
  in which case,
  by \Cref{prop:reduc}, item \ref{prop:reduc_it2}, the problem is
  unbounded and has no solution, thus the output on step \ref{step:unbound} is correct;
  or all the $n-r$ last coefficients of $c'$ are null. In this case the solutions
  are explicitely constructed according to \Cref{solutions}.

  If the condition on step \ref{step:condition} is not satisfied, the problem
  is infeasible by \Cref{prop:reduc} (step \ref{step:core:end}).
  Note that there are at most two recursive calls, thus
  the algorithm terminates in a finite number of steps.
\end{proof}


\begin{remark}
  To maximize the valuation of a linear map over a polyhedron instead of minimizing it 
  (which is equivalent to minimize the absolute value) the same reasoning can be applied. 
  The only difference being that if $\val\left( \lambda\right) < v$ the problem is unbounded
  and as such does not admit a solution. 
\end{remark}

\section{Spectrahedra}
\label{sec:spectrahedra}
  
\subsection{Positive semidefinite matrices}
\newcommand\Mat{Positive semidefinite matrix }
\newcommand\mats{positive semidefinite matrices }
\newcommand\Mats{positive semidefinite matrices }

\begin{definition}
  We say that a matrix $M \in K^{d \times d}$ is {positive semidefinite, or psd} ($M \succeq 0$)
  if all its eigenvalues have nonnegative valuation in an algebraic closure of $K$.
  We denote by $K^{d \times d}_+ \subset K^{d\times d}$ the set of psd matrices.
  The characteristic polynomial of a matrix $M \in K^{d\times d}$ is denoted $\chi_M \in K[T]$.
\end{definition}

Recall that the {Newton polygon} of a polynomial $P = \sum_{i=0}^{d} a_i T^i \in K[T]$ is
defined as the lower convex hull of the set $\{(i, \val a_i) \,|\, 0 \le i \le d\}$, {\it i.e.} the
graph of the greatest convex function $\varphi : [0,d] \to \R$ such that $\varphi(i) \leq \val a_i$
for $i=1,\ldots,d$. This function is piecewise affine and its slopes are exactly the opposites of the valuations
of the roots of $P$.

The following result shows that testing positive semidefiniteness of matrices over
valued fields reduces to checking whether the coefficients of $\chi_M$ are integral.

\begin{theorem}
  \label{caracsdp}
  $M \in K^{d\times d}_+$ if and only if $\chi_M \in \OK[T]$.
\end{theorem}
\begin{proof}
  Both implications will be proven by contraposition.
  Let $M \in K^{d\times d}$ with characteristic polynomial $\chi_M = T^d + \sum_{i=0}^{d-1} a_i T^i$ for $a_0,...,a_{d-1} \in K$.
  First suppose $M \not \in K^{d\times d}_+$, thus $\chi_M$ has a root $\rho$ with $\val\,\rho<0$.
  It follows that
  $$
  + \infty = \val \chi_M (\rho) = \val \left(\rho^d+ \sum_{i=0}^{d-1} a_i \rho^i\right)
  $$ and $\val \rho^d = d \val \rho <0$.
  The properties of the non-archimedean inequality and of the valuation imply that
  $$
  0> \val \rho^d = \val\left(\sum_{i=0}^{d-1} a_i \rho^i\right) \geq
  \min_{0 \leq i \leq d-1} \val \left(a_i \rho^i\right)
  $$
  Therefore, there exists $j \in \left\{0,...,d-1\right\}$ such that $\val(a_j \rho^j) \le \val\rho^d$,
  {\it i.e.} such that $\val a_j \le (d-j) \val \rho <0$. This shows that $\chi_M \not\in \OK[T]$.
  

  Suppose now $\chi_M \not\in \OK[T]$. There exists $j \in \{0,...,d-1\}$ such that $\val a_j <0$.
  The polynomial $\chi_M$ being monic, the point with abscissa $d$ of the Newton polygon of $\chi_M$ is
  $(d,\val(1)) = (d,0)$. 
  Therefore, the point $(j, \val a_j)$ having negative ordinate implies that at least one of the slopes of
  the Newton polygon of $\chi_M$ is positive, so that $\chi_M$ has a root with negative valuation.
\end{proof}

As a consequence, matrices with integral coefficients are psd.

\begin{corollary}\label{cor_caracsdp}
  $\OK^{d\times d} \subset K^{d\times d}_+$.
\end{corollary}
\begin{proof}
  $\OK$ is a ring, therefore $\chi_M \in \OK[T]$ if $M \in \OK^{d \times d}$.
\end{proof}
The reverse inclusion of \Cref{cor_caracsdp} is false in general:
\begin{example}
  Let
  $
  M = \left[\begin{smallmatrix} {3}/{5}+5 & {4}/{5} \\ {4}/{5} & -{3}/{5} \end{smallmatrix}\right].
  $
  The characteristic polynomial of $M$
  is $\chi_M = T^2-5T-4$. However when seen as a matrix over $K=\mathbb{Q}_5$, then $\chi_M \in \mathbb{Z}_5[T] = \OK[T]$. As such $M \in (\mathbb{Q}_5)^{2 \times 2}_+$, while nevertheless $M \not\in (\mathbb{Z}_5)^{2\times 2}$, and hence
  $(\mathbb{Q}_5)^{2 \times 2}_+ \not\subset (\mathbb{Z}_5)^{2 \times 2}$. 
\end{example}

\begin{definition}
  A \emph{cone} in $K^d$ is a set which is closed under multiplication by elements of $\OK$.
\end{definition}

\begin{proposition}
  $K^{d \times d}_+$ is a cone.
  If $\val$ is discrete, then it is both open and closed.
\end{proposition}
\begin{proof}
  The fact that $K^{d \times d}_+$ is a cone is immediate from the definition.
  Next, remark that $\OK[T]$ is closed and open in $K[T]$ equipped with the infinity norm
  $\| \sum_{k=0}^{n} a_k T^k \|_\infty = \max_{0 \leq k \leq n} |a_k|$,
  {where $|\cdot|$ is any norm induced by $\val$}:
  $\OK$ is indeed the open unit ball for $\|\cdot \|_\infty$ that induces a discrete distance for
  which open balls are also closed sets.
  Furthermore, $K^{d\times d}_+ = \chi^{-1}(\OK[T]) $ where $\chi$ is the function that maps a
  matrix to its characteristic polynomial which is continuous.
  Thus $K^{d\times d}_+$ is both open and closed, too.
\end{proof}


\subsection{Definition of spectrahedra}

\begin{definition}
  \label{def_spectrahedra}
  A \emph{spectrahedron in $K^{d\times d}$} is a set of the form
  $$
  \mathcal{S} = \mathcal{L} \cap K^{d\times d}_+ = \left\{X \in K^{d \times d} :
  \exists \, x \in K^n, X = A_0 + \sum_{i=1}^n x_i A_i, X \succeq 0\right\}
  $$
  where $\mathcal{L} = A_0+\span{A_1,\ldots,A_n} \subset K^{d \times d}$ is an affine space, for some
  $A_0,A_1,\ldots,A_n \in K^{d\times d}$.
\end{definition}

In the real case, one often abuses the definition of spectrahedra, saying that the set
$\{x \in \mathbb{R}^n : A_0+\sum_{i=1}^n x_iA_i \succeq 0\}$ is also a spectrahedron.
Remark that this yields that every affine subset of $\mathbb{R}^n$ is also
a spectrahedron: indeed an affine equation $\ell(x) = 0$ is equivalent to the two inequalities
$\ell(x) \leq 0 \leq \ell(x)$. Over a valued field it is necessary to give the following
definition in order to include affine sets in the family of spectrahedra.

\begin{definition}
  \label{def_spectrahedra_Kn}
  A \emph{spectrahedron in $K^n$} is any affine section of a set of the form
  $$
  \mathcal{S}_A = \varphi_A^{-1}(K^{d \times d}_+) = \{x \in K^n : A_0+x_1A_1+\cdots+x_nA_n \succeq 0\}
  $$
  for some $d\in \mathbb{N}$,
  where $\varphi_A : K^n \to \mathcal{L} \subset K^{d \times d}$ is the map
  $\varphi_A(x) = A_0+x_1A_1+\cdots+x_nA_n$, and $\mathcal{L} = A_0+\span{A_1,\ldots,A_n}$.
\end{definition}

Remark that the map $\varphi_A$ in \Cref{def_spectrahedra_Kn} is onto $\mathcal{L}$ but not
injective, unless $A_1,\ldots,A_n$ are linearly independent. Moreover $K^n$ (and thus every affine
subset of $K^n$) is itself a spectrahedron, choosing $A_0 \succeq 0$ and $A_1=A_2=\cdots=A_n=0$.

The class of spectrahedra in $K^{d \times d}$ (or in $K^n$) is closed under intersection: indeed one
has $(\mathcal{L}_1 \cap K^{d \times d}_+) \cap (\mathcal{L}_2 \cap K^{d \times d}_+) =
(\mathcal{L}_1 \cap \mathcal{L}_2) \cap K^{d \times d}_+$
, for two affine spaces
$\mathcal{L}_1 = A_0+\span{A_1,\ldots,A_n}$ and $\mathcal{L}_2 = A_0+\span{B_1,\ldots,B_n}$,
and
$(L_1 \cap \mathcal{S}_A) \cap (L_2 \cap S_B) = (L_1\cap L_2) \cap S_{\diag(A,B)}$ for
two affine spaces $L_1,L_2 \subset K^n$.

\begin{proposition}
  A polyhedron in $K^n$ is a spectrahedron in $K^n$.
\end{proposition}
\begin{proof}
  Assume $\PP = L \cap M$ be a polyhedron in $K^n$, with
  $L = \{x \in K^n : \val \ell_i(x) \geq 0, i=1, \ldots, d\}$
  and $M = \{x \in K^n : \val m_j(x) = +\infty, j=1, \ldots, e\}$,
  for some $d, e \in \mathbb{N}$.
  Let $A_0,\ldots,A_n \in K^{d \times d}$ be s.t.
  $$
  \varphi_A(x) \coloneq A_0+x_1A_1+\cdots+x_nA_n = \diag(\ell_1(x), \ldots, \ell_d(x)).
  $$
  One has $L = \mathcal{S}_A = \varphi_A^{-1}(K^{d \times d}_+)$.
  Since $M \subset K^n$ is affine, then $\PP$ is a spectrahedron in $K^n$.
%
\end{proof}

Spectrahedra form a strict generalization of polyhedra:

\begin{example}
  The following spectrahedron is not a polyhedron:
  \[
  \mathcal{S}_M = \left\{x \in K : M := \begin{bmatrix} 0 & x + \pi^{-1} \\ x & 0 \end{bmatrix} \succeq 0\right\}.
  \] 
\end{example}
\begin{proof}
  Note that
  $\chi_M(T) = T^2-x(x+\pi^{-1})$, therefore by \Cref{caracsdp},
  $\mathcal{S} = \{x \in K : \val(x(x + \pi^{-1})) \geq 0\}$.
  Let $\mathcal{A} = \{x \in K : \val(x) \geq 1\}$ and
  $\mathcal{B} = \{x \in K : \val(x+\pi^{-1}) \geq 1\}$.
  We claim that $\mathcal{S} = \mathcal{A} \cup \mathcal{B}$, which is not
  a ball: therefore, by \Cref{cor:caracdim1}, $\mathcal{S}$ is not a polyhedron.
  We prove now the claim. If $x \in \mathcal{A}$, then $\val(x)\geq 1$ and
  thus $\val(x+\pi^{-1})=-1$, that is $\val(x(x+\pi^{-1})) \geq 0$; similarly
  if $x \in \mathcal{B}$ then $\val(x+\pi^{-1}) \geq 1$ and thus $\val(x) = -1$,
  that is $\val(x(x+\pi^{-1})) \geq 0$; this shows $\mathcal{S} \supset \mathcal{A}
  \cup \mathcal{B}$. Assume now $x \in \mathcal{S} \setminus (\mathcal{A} \cup
  \mathcal{B})$. Then $\val(x)+\val(x+\pi^{-1}) \geq 0$ with $\val(x) \leq 0$
  and $\val(x+\pi^{-1}) \leq 0$, implying $\val(x) = \val(x+\pi^{-1}) = 0$, absurd.
  Therefore $\mathcal{S} \subset \mathcal{A} \cup \mathcal{B}$, that is,
  $\mathcal{S} = \mathcal{A} \cup \mathcal{B}$.
\end{proof}

\begin{definition}
  A subset $S$ of $K^n$ is said \emph{semialgebraic} if it is defined by finite unions and intersections
  of sets defined by equalities of the form $\{x \in K^n  : p(x) = 0\}$
  and of sets defined by inequalities of the form $\{ x\in K^n : \val p(x) \geq 0\}$, for polynomials
  $p \in K[x]$.
\end{definition}

\begin{proposition}
  Spectrahedra in $K^n$ are semialgebraic sets. 
\end{proposition}
\begin{proof}
  Let $\mathcal{S}_A \subset K^n$ be a spectrahedron,
  for some linear matrix $A = A_0+\sum_{i=1}^n x_i A_i$. 
  Denote by $p_i \in K[x]$ the $i$-th coefficient of $\chi_{A}$, $i=0,\ldots,d$.
  By \Cref{caracsdp} one has:
  \begin{equation}\label{reprspec}
    \mathcal{S}_A
    = \left\{x \in K^n : \chi_{A} \in \OK[T]\right\}
  = \bigcap_{i=0}^{d-1} \{x \in K^n : \val(p_i) \geq 0\}.
  \end{equation}
\end{proof}


When $n=1$, we can give a more precise characterization of the geometry of spectrahedra
from the representation in \eqref{reprspec}, using the notion of \emph{diskoid} defined
in Rüth's PhD thesis \cite{ruth_models_2015}.

\begin{definition}
A diskoid of $K$ is a set of the form
\[  \mathrm{Disk}(f,r):= \{ x \in K : \val(f(x)) \geq r \},\]
for some polynomial $f \in K[x]$ and some rational number $r \geq 0.$
\end{definition}

\begin{proposition}
Spectrahedra of $K$ are finite union of balls. \label{prop_spectrahedra_dim1_are_diskoids}
\end{proposition}
\begin{proof}
By \Cref{def_spectrahedra_Kn}, a spectrahedron in $K$ 
can be defined as the set $\mathcal{S}=\{x \in K : A_0+x A_1 \succeq 0\}$
for some integer $d$ and matrices $A_0$ and $A_1$ in $K^{d \times d}$.
Thanks to Theorem \ref{caracsdp},
$\mathcal{S}=\cap_{i=0}^{d-1} \{x \in K : \val(p_i(x)) \geq 0\},$
for some polynomials $p_i$, univariate over $K$.
As such, this set is the intersection of the diskoids
$\mathrm{Disk}(p_i,0).$
By \cite[Lemma 6.9]{ABL_2021}, diskoids are finite union of balls.
Since the intersection of two balls is either empty
or the smallest of the two, we get that intersection of diskoids
are also finite union of balls.
It is thus the case for spectrahedra of $K$.
\end{proof}

\subsection{Polyannuli}


\begin{definition}[Polyannulus]\label{def_annuli}
  A set of the form $\mathcal{C} = \{x \in K \,|\, a \le \val(x) \le b\}$, for $a,b \in \R^*_+ = \{x \in \R : x>0\}$,
  is called an \emph{annulus}. A product $\mathcal{C}_1 \times \mathcal{C}_2 \times \cdots \times \mathcal{C}_n \subset K^n$,
  where each $C_i \subset K$ is an annulus, is called a \emph{polyannulus}.
\end{definition}

\begin{definition}[Semidefinite-representable set]\label{def_semi_def_rep_set}
  A set $\mathcal{R} \subset K^n$ is called
  \emph{semide\-fi\-nite-re\-pre\-sen\-ta\-ble (SDR)} if there exist $m,d \in \N$ and
  matrices $\{A_i,B_j : i = 0, \ldots, n, j = 1, \ldots, m\} \subset K^{d \times d}$ such that
  $$
  \mathcal{R} = \Big\{x \in K^n : \exists y \in K^m, A_0+\sum_{i=1}^n x_iA_i + \sum_{j=1}^m y_jB_j \succeq 0\Big\}
  $$
  that is, $x \in \mathcal{R}$ if and only if there is $y \in K^m$ such that
  $(x,y) \in S_M \subset K^{n+m}$, for some pencil $M(x,y) = A_0+\sum_{i=1}^n x_iA_i + \sum_{j=1}^m y_jB_j$.
  The integer $m$ is called the \emph{height} of the representation and $d$ its \emph{degree}.
\end{definition}

Of course a spectrahedron $\mathcal{S}_A = \varphi_A^{-1}(K^{d \times d})$ is SDR with height $m=0$ and degree
$d$. In order to distinguish between spectrahedra and possibly non-spectrahedral SDR sets
we usually call {spectrahedral shadows} the SDR sets with height $m>0$.

\begin{theorem}
Polyannuli in $K^n$ are spectrahedral shadows with height $m=n$ and degree $d=4n$.
\end{theorem}
\begin{proof}
  Let $\mathcal{C} = \mathcal{C}_1 \times \cdots \times \mathcal{C}_n \subset K^n$ be as in \Cref{def_annuli},
  for some annuli $\mathcal{C}_1,\ldots,\mathcal{C}_n \subset K$. We claim that each $\mathcal{C}_i$ is a spectrahedral
  shadow of height $1$ and degree $4$. Assume $\mathcal{C}_i = \{x_i \in K : a_i \leq \val(x_i) \leq b_i\}$
  for some $a_i, b_i \in \R^*_+$ with $a_i < b_i$. Consider the spectrahedron
  $\mathcal{S}_{M^{(i)}} \subset K^2$ defined by $M^{(i)} \succeq 0$ with
  $$
  M^{(i)}(x_i,y_i) \coloneq \diag\left(\pi^{-a} x_i, \pi^{b} y_i,
  \begin{bmatrix}
    \pi^{-1} & -\pi^{-1} x_i  \\
    \pi^{-1} y_i & - \pi^{-1} \\
  \end{bmatrix}\right).
  $$
  We claim that $x_i \in \mathcal{C}_i$ if and only if there exists $y_i \in K$ such that $(x_i,y_i) \in
  \mathcal{S}_{M^{(i)}}$.
  Indeed, by \Cref{caracsdp}, $M^{(i)}(x_i,y_i) \succeq 0$ if and only if $(x_i,y_i)$ is a solution of the system
  of inequalities:
  \begin{equation}
    \label{eq:annulussyst}
    \begin{cases}
      \val(\pi^{-a} x_i) \ge 0 \\
  \val(\pi^{b} y_i) \ge 0\\
  \val(\pi^{-1}- \pi^{-1}) = +\infty \ge 0\\
  \val(\pi^{-2}\left( x_iy_i-1 \right))  \ge 0 
\end{cases}
\iff	
\begin{cases} 
  \val\left(x_i\right)\ge a\\
  \val\left(y_i\right)\ge -b\\
  \val\left(x_iy_i-1 \right) \ge 2
\end{cases}
\end{equation} 
  The last inequality implies that $\val\left(x_i\right)+\val\left(y_i\right)=\val\left(x_iy_i\right) = \val\left(-1\right) =0$ and therefore that $\val\left(x_i\right)=-\val\left(y_i\right)\le -(-b) = b$, so $x_i \in \mathcal{C}_i$. 
Reciprocally if $x_i \in \mathcal{C}_i$ then $(x_i,x_i^{-1})$ is a solution of \eqref{eq:annulussyst}.
This proves our claim.
Finally, since $\mathcal{C} = \mathcal{C}_1 \times \cdots \times \mathcal{C}_n$, one has
\begin{equation*}
  \begin{aligned}
    \mathcal{C} &=
    \{x \in K^n : \exists y  \in K^n, (x_i,y_i) \in
    \mathcal{S}_{M^{(i)}}, i=1,\ldots,n\} \\
    &= \{x \in K^n : \exists y \in K^n, (x,y) \in \mathcal{S}_M\}
  \end{aligned}
\end{equation*}
with $\mathcal{S}_M = \diag(M^{(1)}, M^{(2)}, \ldots, M^{(n)}) \in K[x,y]^{4n \times 4n}.$
\end{proof}

We prove now that, in general, annuli are not spectrahedra.

\begin{proof}[{Proof of \Cref{main_2}}]
Assume that $\kappa = \OK/\mathfrak{m}_K$, the residue field of $K$, is infinite.
By Proposition \ref{prop_spectrahedra_dim1_are_diskoids} spectrahedra of $K$ are finite union of balls.

In addition, if $\kappa$ is infinite, one can remark that the sphere
$\{x \in K \,|\,  \val(x) =0\}$ is equal to the infinite disjoint union $\amalg_{l \in \kappa} B(b_l, \vert \pi \vert)$ for some representatives $b_l \in \OK$ of the classes of $\OK / \pi\OK = \kappa.$
It thus can not be covered by a finite amount of balls of radius $\vert \pi \vert$
and if one take any ball of radius $1$ centered at a point of the sphere, then one gets the whole
$B(0,1).$

A direct generalization of this argument proves that
if $\kappa$ is infinite, no (non trivial) annulus can be equal to
a finite union of balls. Thus, by Prop. \ref{prop_spectrahedra_dim1_are_diskoids} no (non trivial) annulus can be a spectrahedron.
\end{proof}


\paragraph*{Acknowledgments}
The second author is supported by the ANR Project ``HYPERSPACE'' (ANR-21-CE48-0006-01)
and by INDAM-GNSAGA (2025).

\bibliographystyle{ACM-Reference-Format}
\bibliography{refs}

@book{netzer2023geometry,
  title={Geometry of Linear Matrix Inequalities},
  author={Netzer, T. and Plaumann, D.},
  year={2023},
  publisher={Springer}
}

@article{develin2004tropical,
  title={Tropical convexity},
  author={Develin, M. and Sturmfels, B.},
  journal={Documenta Mathematica},
  volume={9},
  pages={1--27},
  year={2004}
}

@inproceedings{naldi2025solving,
  title={Solving generic parametric linear matrix inequalities},
  author={Naldi, S. and Safey El Din, M. and Taylor, A. and Wang, W.},
  booktitle={Proceedings of the 2025 International Symposium on Symbolic and Algebraic Computation},
  pages={267--276},
  year={2025}
}

@article{renegar1992computational,
  title={On the computational complexity and geometry of the first-order theory of the reals. Part I: Introduction. Preliminaries. The geometry of semi-algebraic sets. The decision problem for the existential theory of the reals},
  author={Renegar, J.},
  journal={Journal of symbolic computation},
  volume={13},
  number={3},
  pages={255--299},
  year={1992},
  publisher={Elsevier}
}

@inproceedings{nemirovski2007advances,
  title={Advances in convex optimization: conic programming},
  author={Nemirovski, A.},
  booktitle={Proceedings of the International Congress of Mathematicians Madrid, August 22--30, 2006},
  pages={413--444},
  year={2007},
  organization={European Mathematical Society-EMS-Publishing House GmbH}
}

@book{engler2005valued,
  title={Valued fields},
  author={A. J. Engler and A. Prestel},
  year={2005},
  publisher={Springer Science \& Business Media}
}

@article{NALDI2018206,
title = {Solving rank-constrained semidefinite programs in exact arithmetic},
journal = {Journal of Symbolic Computation},
volume = {85},
pages = {206-223},
year = {2018},
doi = {https://doi.org/10.1016/j.jsc.2017.07.009},
author = {S. Naldi}
}

@inproceedings{allamigeon2016solving,
  title={Solving generic nonarchimedean semidefinite programs using stochastic game algorithms},
  author={Allamigeon, X. and Gaubert, S. and Skomra, M.},
  booktitle={Proceedings of the 2016 ACM International Symposium on Symbolic and Algebraic Computation},
  pages={31--38},
  year={2016}
}

@book{henrion2020moment,
  title={The Moment-sos Hierarchy},
  author={D.~Henrion and M.~Korda and J-B.~Lasserre},
  volume={4},
  year={2020},
  publisher={World Scientific}
}

@article{allamigeon2019tropical,
  title={The tropical analogue of the Helton--Nie conjecture is true},
  author={Allamigeon, X. and Gaubert, S. and Skomra, M.},
  journal={Journal of Symbolic Computation},
  volume={91},
  pages={129--148},
  year={2019},
  publisher={Elsevier}
}

@article{porkolab1997complexity,
  title={On the complexity of semidefinite programs},
  author={Porkolab, L. and Khachiyan, L.},
  journal={Journal of Global Optimization},
  volume={10},
  number={4},
  pages={351--365},
  year={1997},
  publisher={Springer}
}

@article{henrion2016exact,
  title={Exact algorithms for linear matrix inequalities},
  author={Henrion, D. and Naldi, S. and Safey El Din, M.},
  journal={SIAM Journal on Optimization},
  volume={26},
  number={4},
  pages={2512--2539},
  year={2016},
  publisher={SIAM}
}

@article{allamigeon_tropical_2020,
	title = {Tropical spectrahedra},
	volume = {63},
	issn = {0179-5376, 1432-0444},
	url = {http://arxiv.org/abs/1610.06746},
	doi = {10.1007/s00454-020-00176-1},
	number = {3},
	journal = {Discrete \& Computational Geometry},
	author = {Allamigeon, X. and Gaubert, S. and Skomra, M.},
	month = apr,
	year = {2020},
	note = {arXiv:1610.06746 [math]},
	pages = {507--548},
}

@misc{caruso_computations_2017,
	title = {Computations with p-adic numbers},
	url = {http://arxiv.org/abs/1701.06794},
	doi = {10.48550/arXiv.1701.06794},
	urldate = {2023-07-06},
	publisher = {arXiv},
	author = {Caruso, X.},
	month = jan,
	year = {2017},
	note = {arXiv:1701.06794 [cs, math]},
}

@article{scheiderer_spectrahedral_2017,
  title={Spectrahedral shadows},
  author={Scheiderer, C.},
  journal={SIAM Journal on Applied Algebra and Geometry},
  volume={2},
  number={1},
  pages={26--44},
  year={2018},
  publisher={SIAM}
}

@article{helton_sufficient_2008,
  title={Sufficient and necessary conditions for semidefinite representability of convex hulls and sets},
  author={Helton, J. W. and Nie, J.},
  journal={SIAM J. Optim.},
  volume={20},
  number={2},
  pages={759--791},
  year={2009},
  publisher={SIAM}
}

@inproceedings{ruth_models_2015,
	title = {Models of curves and valuations},
	copyright = {CC BY-SA 3.0 Deutschland},
	url = {https://oparu.uni-ulm.de/xmlui/handle/123456789/3302},
	doi = {10.18725/OPARU-3275},
	language = {en},
	publisher = {PhD Dissertation, Universität Ulm},
	author = {Rüth, J.},
	year = {2015},
	doi = {10.18725/OPARU-3275},
}

@book{Serre:1979,
    AUTHOR = {Serre, J-P.},
     TITLE = {Local fields},
    SERIES = {Graduate Texts in Mathematics},
    VOLUME = {67},
      NOTE = {Translated from the French by Marvin Jay Greenberg},
 PUBLISHER = {Springer-Verlag, New York-Berlin},
      YEAR = {1979},
     PAGES = {viii+241},
      ISBN = {0-387-90424-7},
   MRCLASS = {12Bxx},
  MRNUMBER = {554237 (82e:12016)},
}

@article{ABL_2021,
title = {Minimal pairs, truncations and diskoids},
journal = {Journal of Algebra},
volume = {579},
pages = {388-427},
year = {2021},
issn = {0021-8693},
doi = {https://doi.org/10.1016/j.jalgebra.2021.03.019},
url = {https://www.sciencedirect.com/science/article/pii/S0021869321001654},
author = {Andrei Benguş-Lasnier}
}

\end{document}